\documentclass[smallextended]{svjour3}
\usepackage{graphicx,amsmath,amsfonts}
\usepackage{apacite}
\usepackage{color,soul}
\setstcolor{red}

\smartqed

\newcommand{\Z}{\mathbb Z}

\def\row#1#2{{#1}_1,\ldots ,{#1}_{#2}}

\def\row#1#2{{#1}_1,\ldots ,{#1}_{#2}}

\title{Roughly Weighted Hierarchical Simple Games}

\author{Ali Hameed \and Arkadii Slinko}

\institute{
          Ali Hameed \at
           Department of Mathematics,
           University of Auckland, 
           Auckland, New Zealand. Tel;: +6493737599 ext. 82448; Fax: +6493737457; 
           \email{a.hameed@math.auckland.ac.nz}
           \and
           Arkadii Slinko \at
           Department of Mathematics,
           University of Auckland, 
           Auckland, New Zealand
           \email{a.slinko@auckland.ac.nz}
}

\begin{document}

\maketitle

\begin{abstract}
Hierarchical simple games - both disjunctive and conjunctive - are natural generalizations of  simple majority games. They  take their origin in the theory of secret sharing. Another important generalization of simple majority games with origin in economics and politics are weighted and roughly weighted majority games.  In this paper we characterize roughly weighted hierarchical games identifying  where the two approaches coincide.
\end{abstract}

\section{Introduction}\label{sec:intro}
In both human and artificial societies sometimes  cooperating agents have different 
status with respect to the activity and certain actions are only allowed to coalitions that 
satisfy certain criteria, e.g., to sufficiently large coalitions or coalitions with players of sufficient seniority. Consider, for example, the situation of a 
money transfer from one bank to another. If the sum to be transferred is sufficiently large this transaction 
must be authorised by three senior tellers {\em or} two vice-presidents. However, two senior tellers and a 
vice-president can also authorise the transaction. Another example of a  hierarchical game is the United Nations Security Council, 
where for the passage of a resolution all five permanent members must vote for it, {\em and} also at least nine members in total.

In the theory of secret sharing one of the key concepts is the access structure of a secret sharing scheme which is the set of all coalitions that are authorized to know the secret.  If a smaller coalition knows the secret, then any larger coalition also knows it.  So the mathematical concept that formalises access structures of secret sharing schemes is that of a {\em simple game} introduced by \citeA{vNM:b:theoryofgames}. 

\citeA{Simmons:1988} was the first to introduce  hierarchical simple games (called access structures in secret sharing) in context of secret sharing and recently they have played a significant role in the theory (see, e.g., \citeA{beimel:360,padro:2010}).
Hierarchical access structures introduced by Simmons stipulate that agents are partitioned into $m$ levels, and a sequence 
of thresholds $k_1<k_2<\ldots<k_m$ is set such that, a coalition is authorised if it has either $k_1$ agents of the first level, or $k_2$ agents of the first two levels, or $k_3$ agents of the first three levels etc. 
These hierarchical structures are now called {\em disjunctive}, since only one of the $m$ conditions must be satisfied for a coalition to be authorized. 
Its natural counterpart is a {\em conjunctive} hierarchical access structure \cite{Tassa:2007}, which uses the same thresholds but requires that all $m$ 
conditions must be met for a coalition to be authorized. The money transfer game above was disjunctive, while the United Nations Security Council game was conjunctive. 

Both types of hierarchical games are natural generalizations of simple majority games in which all players have the same status and, for some $k$, any $k$ of them form a winning coalition. 
The theory of simple games has developed several other natural generalizations of simple majority games such as weighted majority games \cite{vNM:b:theoryofgames}, roughly weighted majority games \cite{tz:b:simplegames,GS2011}, and complete simple games \cite{CF:j:complete}. It is important to know how the hierarchical games are related to these well-studied traditional classes of games.

The classification of weighted disjunctive hierarchical games is given in \citeA{beimel:360}. \citeA{gha:t:hierarchical} gave a simple combinatorial proof of this classification and extended this result to conjunctive hierarchical games.  This paper has also established that the class of conjunctive hierarchical games coincides with the class of complete games with minimum \cite{Frei-Pu:2008} which was introduced from purely theoretical considerations.

 This paper gives a complete characterization of both disjunctive and conjunctive hierarchical roughly weighted simple games. We extensively use techniques of the theory of simple games. The main tool  in our characterization is the structural theorem, which describes disjunctive hierarchical games as complete games with a unique shift-maximal losing coalition \cite{gha:t:hierarchical}. Another important result from that paper that will be used, is the duality between disjunctive and conjunctive hierarchical games.
 
\section{Preliminaries}

\subsection{Simple Games and Multisets}
A secret sharing scheme \cite{Simmons:1988,Stinson:1992} stipulates that certain coalitions of users are authorised to know the secret while others are not. It is also required that if a coalition is authorised, then any larger coalition is authorised too. This is formalized in the following definition.
 
\begin{definition}
Let $P=[n]=\{1,2,\ldots, n\}$ be a finite set of players and a collection of subsets $W \subseteq 2^P$ satisfying the following monotonicity condition:
\begin{quote}
if $X\in W$ and $X\subseteq Y$, then $Y\in W$.
\end{quote}
In such case, the pair $G=(P,W)$ is called a {\em simple game}, and the set $W$ is called the set of {\em winning coalitions} of $G$. Coalitions not in $W$ are called {\em losing}.
\end{definition} 

The subset $W $ is completely determined by the set $W_{\text{min}} $ of its minimal winning coalitions.  A player which does not belong to any minimal winning coalitions is called {\em dummy}. He can be removed from any winning coalition without making it losing.

\begin{definition}
A simple game $G=(P,W) $ is called {\em weighted majority game}  if there exist nonnegative weights $\row wn$ and a real number $q$, called {\em quota}, such that 
\begin{equation}
\label{WMG}
X\in W \Longleftrightarrow \sum_{i\in X}w_i\ge q.
\end{equation}
\end{definition}
In secret sharing weighted threshold access structures were introduced by \cite{shamir:1979}. A broader but still well-understood class of games is defined below.

\begin{definition}
A simple game $G$ is called {\em roughly weighted} if there exist non-negative real numbers $\row wn$ and a  real number $q$, called the {\em quota}, not all equal to zero, such that for $X\in 2^P$ the condition $\sum_{i\in X} w_i< q$ implies $X$ is losing, and $\sum_{i\in X} w_i> q$ implies $X$ is winning.  We say that $[q;\row wn]$ is  a {\em rough voting representation} for~$G$.
\end{definition}

We note that in a roughly weighted game nothing can be said about coalitions whose weight is equal to the threshold. There can be both winning and losing ones.

A distinctive feature of many structures is that the set of users is partitioned into subsets and users in each of the subsets have equal status.  In \cite{GS2011,gha:t:hierarchical} we suggested  analyzing such structures with the help of multisets.  Given a simple game $G$ we define a relation $\sim_{G} $ on $P$ by setting $i \sim_{G} j$ if for every set $X\subseteq P$ not containing $i$ and~$j$ 
\begin{equation}
\label{condition}
X\cup \{i\}\in W \Longleftrightarrow X\cup \{j\} \in W.
\end{equation}
We showed that $\sim_{G} $ is an equivalence relation on $P$. Players in the equivalence classes are indistinguishable so the set of players is better viewed in this case as a multiset. Multiset has several types of players with several players in each type.


\begin{definition}
A multiset on the set $[m]$ is a mapping $\mu\colon [m]\to \Z_+$ of $[m]$ into the set of non-negative integers. It is often written in the form
\[
\mu = \{1^{k_1},2^{k_2},\ldots, m^{k_m}\},
\]
where $k_i=\mu(i)$ is called the {\em multiplicity} of $i$ in $\mu$. 
\end{definition}

A multiset $\nu = \{1^{j_1},2^{j_2},\ldots, m^{j_m}\}$ is a submultiset of a multiset $\mu = \{1^{k_1},2^{k_2},\ldots, m^{k_m}\}$, iff $j_i\le k_i$ for all $i=1,2,\ldots, m$. This is denoted as $\nu \subseteq \mu$.\par\smallskip

Given a game $G$ we may also define a relation $\succeq_{G} $ on $P$ by setting $i \succeq_{G} j$ if $
X\cup \{j\}\in W $ implies $ X\cup \{i\} \in W
$
for every set $X\subseteq U$ not containing $i$ and~$j$. 
It is known as Isbell's desirability relation \cite{tz:b:simplegames}. The game is called {\em complete} if $\succeq_{G}$ is a total  order. We also define  $i \succ_{G} j$ as $i \succeq_{G} j$ but not $j \succeq_{G} i$. \par\smallskip

The existence of large equivalence classes relative to $\sim_G$ allows us to compress the information about the game. This is done by the following construction.
Let now $G=(P,W)$ be a game and $\sim_{G}$ be its corresponding equivalence relation. Then $P$ can be partitioned into a finite number of equivalence classes  $P=P_1\cup P_2\cup\ldots\cup P_m$ and suppose $|P_i|=n_i$. Then we put in correspondence to the set of players $P$ a multiset $\bar{P}=\{1^{n_1},2^{n_2},\ldots, m^{n_m}\}$. 
We carry over the game structure to $\bar{P}$ as well by defining the set of winning submultisets $\bar{W}$ by assuming that a submultiset  $Q=\{1^{\ell_1},2^{\ell_2},\ldots, m^{\ell_m}\}$ is winning in $\bar{G}$ if a subset of $P$ containing $\ell_i$ players from $P_i$ for every $i=1,2,\ldots, m$, is winning in~$G $. This definition is correct since the sets $P_i$ are defined in such a way that it does not matter which $\ell_i$ users from $P_i$ are involved. We call $\bar{G}=(\bar{P},\bar{W})$ the canonical representation of $G$.
The relations $\succeq_G$,  $\sim_G$,  $\succ_G$ induce the respective relations $\succeq_{\bar{G}}$,  $\sim_{\bar{G}}$,  $\succ_{\bar{G}}$. We also note that  $1\succ_{\bar{G}} 2\succ_{\bar{G}}\ldots \succ_{\bar{G}} m$.

\begin{definition}
A pair $G=(P,W)$ where $P = \{1^{n_1},2^{n_2},\ldots, m^{n_m}\}$ and  $W$ is a  system of submultisets of the multiset $P$ is said to be a {\em simple game} on a multiset of players $P$ if $X\in W $ and $X\subseteq Y$ implies $Y\in W$.  Submultisets of $P$ we will call {\em coalitions}.
\end{definition}

\begin{definition}
We say that $G=(P,W)$ on a multiset of players $P$ is a weighted majority game if there exist non-negative weights $\row wm$ and $q\ge 0$ such that  a coalition $Q=\{1^{\ell_1},2^{\ell_2},\ldots, m^{\ell_m}\}$ is winning in $G $ iff $\sum_{i=1}^m \ell_iw_i\ge q$.
\end{definition}

It is a well-known fact that any weighted game can be given a voting representation in which players of equal Isbell's desirability have equal weights \cite{tz:b:simplegames}. However we need a similar statement that would be also applicable to roughly weighted games.

\begin{lemma}
\label{equal_weights}
A simple game  $G=(P,W)$ is a roughly weighted majority game if and only if the corresponding simple game $\bar{G}=(\bar{P},\bar{W})$ is. 
\end{lemma}

\begin{proof}
Suppose there are  $m$ equivalence classes $\row Pm$ of players and let us denote $[i]$ the equivalence class to which $i$ belongs. The statement is nontrivial only in one direction. The nontrivial part is to prove that if the game $G$ on $P$ is roughly weighted, then the game $\bar{G}$ on $\bar{P}$ is also roughly weighted. So suppose that there exists a system of weights $\row wn$ and the quota $q\ge 0$, not all equal to zero, such that    $\sum_{i\in X}w_i > q$ implies $X\in W$ and  $\sum_{i\in X}w_i < q$ implies $X\in L$. Our statement will be proved if we can find another system of weights $\row un$ for $G $ which satisfy two conditions:
 \begin{enumerate}
 \item[(i)] $u_i=u_j$ if $[i]=[j]$,
 \item[(ii)]  ${\sum_{i\in X}u_i > q}$ implies $X \in W$.
  \item[(iii)]  ${\sum_{i\in X}u_i < q}$ implies $X \notin W$.
 \end{enumerate}
 
 We define this alternative system of weights by setting
$
 u_i=\frac{1}{|[i]|}\sum_{j\in [i]} w_j,
$
i.e., we replace the weight of $i$th user with the average weight of users in the equivalence class  to which $i$ belongs. It obviously satisfies (i). Let us prove that it satisfies (ii).

Let $X\subseteq P $ and ${\sum_{i\in X}u_i > q}$. Let $k_i=|X\cap P_i|$. Let $X^+$ be the subset of $P$ which results in replacing in $X$, for all $i=1,2,\ldots, m$, all $k_i$ elements of $P_i$ with the ``heaviest'' elements from the same class. Then the weight of $X^+$ relative to the old system of weights is greater or equal to $\sum_{i\in X}u_i $ and hence greater than $q$. So $X^+$ is winning in $G$, and so is $X$, because we replaced all elements with  equivalent ones. (iii) is proved similarly.
\end{proof}

\begin{definition}[Hierarchical Games] 
\label{HG}
Suppose the set of players $P$ is partitioned into $m$ disjoint subsets $P=\cup_{i=1}^m P_i$ and let $k_1<k_2<\ldots<k_m$ be a sequence of positive integers. Then we define the game $H_\exists=(P,W)$ by setting
\[
W = \{ X\in 2^P\mid \exists i \left(\left|X\cap \left(\cup_{j=1}^i P_i\right)\right|\ge k_i\right) \}.
\]
and call it a {\em disjunctive hierarchical game}. For a sequence of thresholds $k_1<\ldots <k_{m-1}\le k_m$ we may define
\[
W = \{ X\in 2^P\mid \forall i \left(\left|X\cap \left(\cup_{j=1}^i P_i\right)\right|\ge k_i\right) \}.
\]
We call the resulting game a {\em conjunctive hierarchical game} $H_\forall=(P,W)$.
\end{definition} 

From this definition it follows that any hierarchical game is complete with $p\succeq_Hq$ iff $p\in P_i$, $q\in P_j$ and $i<j$. However, for arbitrary values of parameters $\row km$ and $\row nm$, where $n_i=|P_i|$, we cannot guarantee that the multiset representation $\bar{H}$ of $H$ will be the defined on the multiset $\bar{P}=\{1^{n_1},2^{n_2},\ldots, m^{n_m}\}$, since it may be possible to have $u\sim_{H}v$ for some $u\in P_i$ and $v\in P_j$ for $i\ne j$. So we can guarantee
\begin{equation}
\label{non-strict12...m}
1\succeq_H2\succeq_H3\succeq_H\ldots\succeq_Hm
\end{equation}
but cannot guarantee that these inequalities are strict. Some distinct classes of the partition may collapse into a single class. When this does not happen the representation is called {\em canonical}. 

The necessary and sufficient conditions for a representation to be canonical are obtained in  \cite{gha:t:hierarchical}. They are given in the theorem below.

\begin{theorem}
\label{Thm_1}
Let $H$ be a disjunctive  hierarchical game defined on the set of players $P$ partitioned into $m$ disjoint subsets $P=\cup_{i=1}^m P_i$, where $n_i=|P_i|$, by a sequence of positive thresholds $k_1<k_2<\ldots<k_m$.  Then the multiset representation $\bar{H}$ of $H $  is defined on $\bar{P}= \{1^{n_1},2^{n_2},\ldots, m^{n_m}\}$ if and only if
\begin{enumerate}
\item[(a)] $k_1\le n_1$, and
\item[(b)] $k_i <  k_{i-1}+n_i$  for every $1<i< m$.
\end{enumerate}
When (a) and (b) hold the sequence  $(\row k{m-1})$ is determined uniquely. Moreover, $H$ does not have dummies if and only if $k_m <  k_{m-1}+n_m$; in this case  $k_m$ is determined uniquely as well. When $k_m \ge  k_{m-1}+n_m$ the last level consists entirely of dummies and in this case we set $k_m =  k_{m-1}+n_m$.  $H$ has passers only in case $k_1=1$, when all players of level 1 are passers.\par
Moreover, for each $i\in \{1,2,\ldots, m-1\}$ there exists a minimal winning coalition of size $k_i$ that is contained in $ \{1^{n_1},\ldots, i^{n_i}\}$ but not in $ \{1^{n_1},\ldots, (i-1)^{n_{i-1}}\}$. If $H$ has no dummies, this is also true for $i=m$.
\end{theorem}

For a conjunctive hierarchical game there are also $m$ thresholds $k_1<\ldots <k_{m-1} \le k_m$ (note the possibility of having an equality between $k_{m-1}$ and $k_m$). The same conditions (a) and (b) are necessary and sufficient too. In this case the  last level consists entirely of dummies if and only if $k_{m-1}=k_m$ and the first level consists of blockers if and only if $k_1=n_1$. \par\medskip

By $H_\exists({\bf n},{\bf k})$ and $H_\forall({\bf n},{\bf k})$ we will denote the $m$-level hierarchical access structure canonically represented by  ${\bf n}=(\row nm)$ and ${\bf k}=(\row km)$ for which the conditions (a) and (b) hold. 

The following corollary from Theorem~\ref{Thm_1} proved by  \citeA{gha:t:hierarchical} will be also important later.

\begin{corollary}
\label{n_2>1}
Let $G=H_\exists({\bf n},{\bf k})$ be a hierarchical game with $m$ levels in its canonical representation. Then  we have $n_i>1$  for every $1<i< m$.
\end{corollary}

\subsection{Subgames, Reduced Games and Duality}

\begin{definition}
Let $G=(P,W)$ be a simple game with $A\subseteq P$. Let us define subsets $W_{\text{sg}}\subseteq W$ and $W_{\text{rg}}\subseteq W$ by
\[
W_{\text{sg}}=\{X\subseteq  A^c\mid X\in W\}, \  
W_{\text{rg}}=\{X\subseteq A^c\mid X\cup A\in W\},
\]
where $A^c=P\setminus A$. Then the game $G_A=(A^c,W_\text{sg})$ is called a {\em subgame} of $G$ and $G^A=(A^c,W_\text{rg})$ is called a {\em reduced game} of $G$. Any game $H$ that is obtained from $G$ by a sequence of operations of taking subgame or a reduced game is called a {\em minor}.
\end{definition}

Let us now briefly recap the concept of duality in games. The {\em dual game} of a
game $G=(P,W)$ is defined as $G^{*} = (P,L^{c})$. Equivalently,
the winning coalitions of the game $G^{*}$ dual to $G$ are
exactly the complements of losing coalitions of $G$. We have $G=G^{**}$.
%
We also note that if $A \subseteq P$, then $(G_A)^{*} = (G^{*})^A$ and $(G^A)^{*} = (G^{*})_A$. Moreover, the operation of taking the dual is known to  preserve weightedness and rough weightedness (\citeA{tz:b:simplegames}, Proposition 4.10.1(i), p. 166). We need to be a bit more precise here.

\begin{lemma}
\label{subs_and_reducs}
Let $G=(P,W)$ be a roughly weighted game with rough voting representation $[q;\row wn]$. Suppose that $A\subseteq P$ such that $w(b)>0$ for some $b\in A^c$. Then the subgame $G_A$ and the reduced game $G^A$ are roughly weighted. In rough voting representations of $G_A$ and $G^A$ the weights of players are the same as in $G$  and quotas are $q$ and $\max(0, q-w(A))$, respectively. 
\end{lemma}

\begin{proof}
The subgame $G_A$ is realised as a roughly weighted game by the restriction $w|_{A^c}$ together with the original quota. The reduced game $G^A$ is realised as a roughly weighted game by the restriction $w|_{A^c}$ together with the quota $q'=\max(0, q-w(A))$. 
\end{proof}

We will also use the fact that Isbell's desirability relation is self-dual, that is $x\succeq_G y$ if and only if
$x\succeq_{G^{*}} y$ \cite{tz:b:simplegames}.
All these concepts can be immediately reformulated for the games on multisets if we define the complement $X^c$ of a submultiset $X=\{1^{\ell_1}, \ldots, m^{\ell_m}\}$ in $P=\{1^{n_1},\ldots, m^{n_m}\}$ as
the submultiset
\[
X^c=\{1^{n_1-\ell_1}, \ldots, m^{n_m-\ell_m}\}.
\]

Let us now introduce the following notation. Let ${\bf n}=(\row nm)$ be a fixed vector of positive integers. Then for any another such vector  ${\bf k}=(\row km)$ such that conditions (a) and (b) of Theorem~\ref{Thm_1} are satisfied we define the vector
\begin{equation}
\label{transform}
{\bf k}^{*}=(n_1-k_1+1, n_1+n_2 - k_2+1, \ldots, \sum_{i \in [m]} n_i -k_m +1).
\end{equation}
Note that ${\bf k}^{{*}{*}}={\bf k}$. 

\begin{theorem}[\citeA{gha:t:hierarchical}]
\label{all-exists-dual}
Let $H=H_\exists({\bf n},{\bf k})$ be an $m$-level hierarchical disjunctive game. 
Then the game  dual to $H$ will be the conjunctive hierarchical game  $H^{*}=H_\forall({\bf n}, {\bf k}^{*})$. Similarly, if 
$H=H_\forall({\bf n},{\bf k})$ is an $m$-level hierarchical conjunctive game, then $H^{*}=H_\exists({\bf n}, {\bf k}^{*})$.
In particular,
\[
H_\exists({\bf n},{\bf k})^{*}=H_\forall({\bf n}, {\bf k}^{*}),\qquad H_\forall({\bf n},{\bf k})^{*}=H_\exists({\bf n}, {\bf k}^{*}).
\]
\end{theorem}

\subsection{Weighted Disjunctive and Conjunctive Hierarchical Games}

Firstly, we will describe disjunctive hierarchical weighted games. Beimel, Tassa and Weinreb \citeyear{beimel:360} characterized these without dummies as part of their characterization of ideal weighted threshold secret sharing schemes.  However, their proof is indirect and heavily relies upon the connection between ideal secret sharing schemes and matroids. The following slightly more general theorem is proved by \citeA{gha:t:hierarchical} using a simple and purely combinatorial argument.

\begin{theorem}
\label{BTW}
Let $G=H_\exists({\bf n},{\bf k}) $ be an $m$-level disjunctive hierarchical simple game. For $m>1$ we define the subgame $H_\exists ({\bf n}',{\bf k}')$ where ${\bf n}'=(\row n{m{-}1})$ and ${\bf k}'=(\row k{m{-}1})$. Then $G $ is a weighted majority game iff one of the following conditions is satisfied:
\begin{enumerate}
\item[(1)] $m=1$;
\item[(2)] $m=2$ and $k_2=k_1+1$;
\item[(3)] $m=2$ and $n_2= k_2-k_1+1$;
\item[(4)] $m\in \{2,3\}$ and $k_1=1$. When $m=3$, $G$ is weighted if and only if the subgame $H_\exists({\bf n}',{\bf k}')$  falls under (2) or (3);
\item[(5)] $m\in \{2,3,4\}$,  $k_m = k_{m-1}+n_m$,  and the subgame $H_\exists ({\bf n}',{\bf k}')$ falls under one of the conditions (1) -- (4). 
\end{enumerate}
\end{theorem} 

We note that the only case when we can have four levels is when the top level and the bottom one are both trivial, that is $k_1=1$ and $k_4= k_3+n_4$.  \cite{beimel:360} do not get four levels (only three) as they allow trivial levels of one kind but not of another. By duality \cite{gha:t:hierarchical}, we obtain a similar characterization in the conjunctive case.

\begin{theorem}
\label{conj-char}
Let $G=H_\forall({\bf n},{\bf k}) $ be an $m$-level conjunctive
hierarchical simple game. Then $G $ is a weighted majority game iff one of
the following conditions is satisfied:
\begin{enumerate}
\item[(1)] $m=1$ and $G$ is a simple majority game;
\item[(2)] $m=2$ and $k_2=k_1+1$;
\item[(3)] $m=2$ and $n_2= k_2-k_1+1$;
\item[(4)] $m\in \{2,3\}$ and $k_1=n_1$,  that is, the game has two or three levels and the first one consists entirely of blockers. In case $m=3$,  the reduced game $H_\forall({\bf n},{\bf
k})^{\{1^{n_1}\}}=H_\forall({\bf n}',{\bf k}')$ of $G$, where ${\bf n}'=(n_2,n_3)$ and
${\bf k}'=(k_2-k_1,k_3-k_1)$, falls under (2) or (3);
\item[(5)] $m\in \{2,3,4\}$ with $k_m = k_{m-1}$, that is the game has up to four levels but the last one consists entirely of dummies. Moreover, the reduced game
$H_\forall^{\{m^{n_m}\}}({\bf n},{\bf k})=H_\forall({\bf n}',{\bf k}')$,
where ${\bf n}'=(n_1,\ldots, n_{m-1})$ and ${\bf k}'=(k_1,\ldots,
k_{m-1})$, falls under one of the (1) -- (4).
\end{enumerate}
\end{theorem}

\subsection{Complete Games and Structural Theorems}

 Suppose a game $G $ is complete. By a {\em shift} we mean a replacement of a player of a coalition by a less influential player which did not belong to it.  Formally, given a  coalition $X$, player $i\in X$ and another player $j\notin X$ such that $j\prec_{G}i$ we say that the coalition $ (X\setminus \{i\})\cup \{j\} $ is obtained from $X$ by a shift. A winning coalition $X$ is {\em shift-minimal} if every coalition contained in it and every coalition obtained from it by a shift are losing. A losing coalition $Y$ is said to be {\em shift-maximal} if every coalition that contains it is winning and there does not exist another losing coalition from which $Y$ can be obtained by a shift.  These concepts can be immediately reformulated for games  on multisets. 

\begin{definition}
Let $G$ be a complete simple game on a multiset $P=\{1^{n_1},\ldots, m^{n_m}\}$ with 
$
1\succ_{G}2\succ_{G}\ldots \succ_{G} m.
$
Suppose a submultiset
$
A'=\{\ldots,i^{\ell_i},\ldots, j^{\ell_j},\ldots \}
$
has $\ell_i\ge 1$ and $\ell_j<n_j$ for some $i<j$. We say that the submultiset
$
A'=\{\ldots,i^{\ell_i-1},\ldots, j^{\ell_j+1},\ldots \}
$
is obtained from $A$ by a shift. 
\end{definition}
Shift-minimal winning and shift-maximal losing coalitions are then defined straightforwardly. The following theorem  \cite{gha:t:hierarchical} will play a crucial role in this paper. 

\begin{theorem}
\label{shift-maximal-losing}
The class of disjunctive hierarchical simple games is exactly the class of complete games with a unique shift-maximal losing coalition.  The class of conjunctive hierarchical simple games is exactly the class of complete games with a unique shift-minimal winning coalition. 
\end{theorem}

\section{ Roughly Weighted Disjunctive Hierarchical Games. First Results}

\subsection{Minors of Disjunctive Hierarchical Games}

In proving our classification we will work with hierarchical disjunctive games and obtain the result for hierarchical conjunctive games by duality. Hence in this section we restrict ourselves with disjunctive case only.
The following statements are easy to check.

\begin{proposition}
\label{cut_tail}
Let  ${\bf n}'=(\row n{m{-}1})$, ${\bf k}'=(\row k{m{-}1})$. Then $H'=H_\exists({\bf n}',{\bf k}')$ is a subgame of $G=H_\exists({\bf n},{\bf k})$. This subgame never has dummies and it does not have passers if $G$ did not.
\end{proposition}

\begin{proof}
Indeed, $H_\exists({\bf n}',{\bf k}')=G_A$ for $A=\{m^{n_m}\}$. By Theorem~\ref{Thm_1} there always exists a minimal winning coalition which is contained in $H'$ and has a nonempty intersection with the $(m-1)$th level. Hence the players of $(m-1)$th level are not dummies.
\end{proof}

\begin{proposition}
\label{cut_head}
Let  ${\bf n}'=(n_2{+}k_1{-}1,n_3,\ldots,n_m)$ and ${\bf k}'=(k_2,\ldots, k_m)$. Then $H_\exists({\bf n}',{\bf k}')$ is a subgame of $G=H_\exists({\bf n},{\bf k})$. This subgame does not have passers and it has dummies if and only if $G$ had. 
\end{proposition}

\begin{proof}
Indeed, $H_\exists({\bf n}',{\bf k}')=G_A$ for $A=\{1^{n_1-k_1+1}\}$.  If we make $n_1-k_1+1$ elements of level one unavailable the first constraint loses any bite and the first level collapses.
\end{proof}

\begin{lemma}
\label{two_consecutive_k}
For any $i=1,2,\ldots, m-1$ there exists $n_i'$, such that for ${\bf n}'= (n_i',n_{i+1})$ and ${\bf k}'=(k_i,k_{i+1})$ the game $G'=H_\exists({\bf n}',{\bf k}')$ is a subgame of $G=H_\exists({\bf n},{\bf k})$.
\end{lemma}

\begin{proof}
Follows directly from Propositions~\ref{cut_tail} and~\ref{cut_head}.
\end{proof}


\begin{proposition}
\label{remove_one}
Let $G=H_\exists({\bf n},{\bf k})$, where ${\bf n}=(\row nm)$, ${\bf k}=(k_1,\ldots, k_m)$. Suppose that  $k_{i}>k_{i-1}+1$ for some $i\in \{1,\ldots, m\}$. Then for 
\begin{align*}
&{\bf n}'=(n_1,\ldots,n_{i-1},n_i-1,n_{i+1},\ldots,n_m),\\ 
&{\bf k}'=(k_1,\ldots,k_{i-1},k_i-1,k_{i+1}-1,\ldots,k_m-1)
\end{align*}
$G'=H_\exists({\bf n}',{\bf k}')$ is a reduced game of $G$.
Moreover, if $G$ did not have dummies, then $G'$ would not have them. 
\end{proposition}

\begin{proof}
Since all representations are canonical,  the condition $k_{i}>k_{i-1}+1$ implies  that $n_i>k_i-k_{i-1}\ge 2$, so $n_i\ge 3$. We note now that $G'=G^A$ for $A=\{i\}$. It is easy to check that all conditions (a) and (b) are satisfied for the new values of parameters ${\bf n}'$ and ${\bf k}'$. 
\end{proof}

Let us now generalise Theorem~\ref{BTW} and classify roughly weighted disjunctive hierarchical games. Considering roughly weighted hierarchical $m$-level game $H$ it will be convenient to have the quota equal to 1. Also by Lemma~\ref{equal_weights} we may consider that all players of level~$i$ have weight $w_i$ so that any rough voting representation has the form $[1;\row wm]$. If $X$ is a coalition of $H$, by $w(X)$ we will denote the total weight of~$X$. \par\medskip

We will use the geometric approach based on Theorem~\ref{shift-maximal-losing} and need the following observation. 

\begin{proposition}
\label{maxvalue}
Let $H$ be a disjunctive hierarchical game and let $M$ be its unique shift-maximal losing coalition. Suppose $H$ is roughly weighted with rough voting representation $[1;\row wm]$. 
Then $w(M)\ge w(L)$ for any losing coalition $L$. 
\end{proposition}

\begin{proof}
Since any shift replaces a player with a less influential one, the weight of the latter must be not greater than the weight of the former. This secures that if a coalition $S$ obtained from a coalition $T$ by a  shift, then $w(T)\ge w(S)$. If $S$ is a subset of $T$, then also $w(T)\ge w(S)$. 
Since $M$ is a unique shift-maximal losing coalition we will have $w(M)\ge w(L)$ for any losing coalition $L$.
\end{proof}

This simple proposition has a useful corollary.

\begin{corollary}
\label{w(M)=q}
 Let $H$ be a disjunctive hierarchical game and let $M$ be its unique shift-maximal losing coalition. Suppose $H$ is roughly weighted with rough voting representation $[1;\row wn]$ but not weighted. 
Then $w(M)=1$.
\end{corollary}

\begin{proof}
If $w(M)<1$, then by Proposition~\ref{maxvalue} there is no losing coalitions on the threshold. In this case the game is weighted.
\end{proof}

The following will also be very useful.
\begin{proposition}
\label{allnonzero}
Let $H_\exists({\bf n},{\bf k})$ be the $m$-level disjunctive hierarchical game with no passers and no dummies. Suppose it is roughly weighted with rough voting representation $[1;\row wn]$. Then 
\begin{enumerate}
\item[(i)] $w_1\ge w_2\ge \ldots \ge w_m$.
\item[(ii)] $w_{i}>0$ for $i=1,2,\ldots, m-1$.
\end{enumerate}
\end{proposition}

\begin{proof}
As there are no dummies, $k_m-k_{m-1}<n_m$ is satisfied. We also have $k_1>1$ as no passers are present. By Theorem~\ref{shift-maximal-losing} we know that $H_\exists({\bf n},{\bf k})$ has a unique shift-maximal losing coalition. This coalition then would be
\begin{equation}
\label{shift_maximal}
M=\{1^{k_1-1},2^{k_2-k_1},\ldots, m^{k_{m}-k_{m-1}}\}.
\end{equation}
By Corollary~\ref{w(M)=q} we have
\[
w(M)=(k_1-1)w_1+(k_2-k_1)w_2+\ldots+(k_m-k_{m-1})w_m = 1. 
\]
If only $w_{i+1}>w_i$, then 
\begin{align*}
w(M)\ge &(k_1-1)w_1+\ldots +(k_i-k_{i-1})w_i+(k_{i+1}-k_{i})w_{i+1}>\\
& (k_1-1)w_1+\ldots +(k_{i-1}-k_{i-2})w_{i-1}+(k_{i+1}-k_{i-1})w_{i}\ge 1,
\end{align*}
since the latter is the weight of a winning coalition $\{ 1^{k_1-1},2^{k_2-k_1},\ldots, i^{k_{i+1}-k_{i-1}} \}$ (indeed the cardinality of this multiset is $k_{i+1}-1\ge k_i$). This contradiction proves (i).\par
 
To prove (ii) we note that by Theorem~\ref{Thm_1} we have $k_i-k_{i-1}<n_i$, and hence every level in multiset $M$ is not completely filled and has some capacity. Suppose first that $k_m-k_{m-1}\le n_m-2$. Then the multiset 
\begin{equation*}
\label{modified_shift_maximal}
M'=\{1^{k_1-2},2^{k_2-k_1},\ldots, m^{k_{m}-k_{m-1}+2}\}
\end{equation*}
is winning from which we see that $w_m>0$. If $k_i-k_{i-1}=n_i-1$, then the multiset
\begin{equation*}
\label{modified_shift_maximal}
M''=
\{1^{k_1-2},2^{k_2-k_1},\ldots, (m-1)^{k_{m-1}-k_{m-2}+1}, m^{k_{m}-k_{m-1}+1}\}
\end{equation*}
is winning whence $w_{m-1}>0$. This proves~(ii).
\end{proof}

The two following results will be very useful later on in characterising roughly weighted hierarchical simple games with three levels and more.
\begin{lemma}
\label{lastweightzero}
Suppose that an $m$-level disjunctive hierarchical game $H=H_\exists({\bf n},{\bf k})$ without passers and without dummies is roughly weighted with rough voting representation $[1;\row wm]$ but not weighted.  If the subgame $H'=H_\exists({\bf n}',{\bf k}')$, where ${\bf n}'=(n_1,\ldots,n_{m-1})$ and ${\bf k}'=(k_1,\ldots,k_{m-1})$,  is also not  weighted, then  $w_m=0$. 
\end{lemma}

\begin{proof}
By Proposition~\ref{cut_tail}  $H'$ is a subgame of $H$, hence it is roughly weighted with rough voting representation $[1;\row w{m{-}1}]$. By Proposition~\ref{allnonzero} all the weights are nonzero.  The shift-maximal losing coalition $M$ for $H$ will be (\ref{shift_maximal}) and for $H'$ it will be
\[
M'=\{1^{k_1-1},2^{k_2-k_1},\ldots, (m-1)^{k_{m-1}-k_{m-2}}\}.
\]
If the game $H'$ is not weighted, then by Corollary~\ref{w(M)=q} we have $w(M')=1$. As $1=w(M)=w(M')+(k_m-k_{m-1})w_m$ and $k_m>k_{m-1}$, this implies $w_m=0$.
\end{proof}

\begin{corollary}
\label{L4}
Suppose that a four-level disjunctive hierarchical game $H=H_\exists({\bf n},{\bf k})$ without passers and dummies is roughly weighted with rough voting representation $[1;w_1,w_2,w_3,w_4]$. Then $w_4=0$, and the subgame $H_\exists({\bf n}',{\bf k}')$, where ${\bf n}'=(n_1,n_2,n_3)$ and ${\bf k}'=(k_1,k_2,k_3)$, is roughly weighted with positive coefficients in its rough voting representation.
\end{corollary}

\begin{proof}
The game $H_\exists({\bf n}',{\bf k}')$  is roughly weighted and it is not weighted by Theorem~\ref{BTW}. By Lemma~\ref{lastweightzero} we get then $w_4=0$. By Proposition~\ref{allnonzero} each of the  weights $w_1,w_2,w_3$ is nonzero.
\end{proof}

The reader might expect that $w_m=0$ implies that the $m$th level must consist of dummies. In a roughly weighted game this may not be the case and here is an example illustrating this.

\begin{example}
Let us consider disjunctive hierarchical game $H=H_\exists({\bf n},{\bf k})$ with ${\bf n}=(3,3,3)$ and ${\bf k}=(2,3,5)$. It is roughly weighted relative to the weights $[1; \frac{1}{2},\frac{1}{2},0]$.  Indeed, the shift-minimal winning coalitions of $H$ are $\{1^2\}$, $\{2^3\}$, $\{2^2,3^3\}$. They all have weight at least 1. The unique shift-maximal losing coalition $\{1,2,3^2\}$ also has weight~$1$ but this is allowed. The players of the third level are not dummies despite having weight 0. Moreover in any other system of weights consistent with the game $H$, players of level three will have weight~$0$.  
\end{example} 

\begin{proof}
Let us prove the last statement about the game in this example. If $[1;w_1,w_2,w_3]$ is any rough voting representation for $H$, then  the following system of inequalities must hold:
\begin{eqnarray}
w_1&\ge& \frac{1}{2},\\
w_2&\ge& \frac{1}{3},\\
2w_2+3w_3&\ge& 1,\\
w_1+w_2+2w_3&=& 1.
\end{eqnarray}
However (5) and (8) imply $w_2+2w_3\le \frac{1}{2}$, which implies $2w_2+4w_3\le 1$, which together with (7) implies $w_3=0$.
\end{proof}

Now we can restrict the number of nontrivial levels to four.

\begin{lemma}
\label{fourlevels}
A roughly weighted $m$-level disjunctive hierarchical game $H=H_\exists({\bf n},{\bf k})$ without passers and without dummies may have no more than four levels, i.e.,  $m\le 4$. 
\end{lemma}

\begin{proof}
Suppose $m\ge 5$. Consider the game $H'=H_\exists({\bf n}',{\bf k}')$, where ${\bf n}'=(\row n{m-1})$ and ${\bf k}'=(\row k{m-1})$. By Proposition~\ref{cut_tail} $H'$ is a subgame of $H$ and has no passers or dummies. By Lemma~\ref{subs_and_reducs} it is roughly weighted. As it has four or more levels, by Theorem~\ref{Thm_1} it is not weighted.  By Proposition~\ref{allnonzero} we have $w_{m-1}>0$, but we also have $w_{m-1}=0$ by Lemma~\ref{lastweightzero} applied to $H'$. This contradiction proves the lemma.
\end{proof}

We will see in the following section that four nontrivial levels are also not achievable.

\section{A  Characterization of Roughly Weighted Disjunctive Hierarchical Games}

Now we can start our full characterisation of all roughly weighted hierarchical games.  Due to the results of the previous section our main focus will be on $2$-level ones,  then $3$-level ones, and then showing that the fourth level may not be added unless we allow dummies or passers. 

\subsection{Some general  comments}
Here is our general strategy to analyse if a particular disjunctive hierarchical game $G=H_\exists({\bf n},{\bf k})$ is roughly weighted but not weighted. Firstly we list all shift-minimal winning coalitions and write a system of inequalities in $\row wn$ that is equivalent to the fact that in the game with rough voting representation $[1;\row wn]$ these coalitions are above or on the threshold. Requiring that shift-minimal winning coalitions are on or above the threshold is sufficient for ensuring that all winning coalitions are on or above the threshold.   This is due to the fact that every shift reduces the weight and adding players does not decrease the weight.
By Theorem~\ref{shift-maximal-losing} there is a unique shift-maximal losing coalition $M$. So then (assuming no dummies)  we by Corollary~\ref{w(M)=q} we have the following equation
\begin{equation}
\label{shift_maximal_weight}
(k_1-1)w_1+(k_2-k_1)w_2+\ldots+(k_n-k_{n-1})w_n = 1.
\end{equation} 
 $M$ is exactly on the threshold.  The so-composed system has a solution if and only if the game is roughly weighted. 

The possible shift-minimal winning coalitions in a two- or a three-level disjunctive hierarchical game $G=H_\exists ({\bf n},{\bf k})$ and the inequalities corresponding to them are as follows: 
\begin{itemize}
\item When $k_i \leq n_i$,  we have  shift-minimal winning coalition $\{i^{k_i}\}$ and the corresponding inequality
\begin{equation}
\label{n1winning}   
k_iw_i \geq 1. 
\end{equation}  

\item In the case when $k_2 > n_2$ the coalition $\{1^{k_2-n_2},2^{n_2}\}$ is a shift-minimal winning coalition, then we have
\begin{equation}
\label{n2winning}   
(k_2-n_2)w_1 + n_2w_2 \geq 1. 
\end{equation}  

\item In the case when $k_3 > n_3$ there are two possibilities, either $k_2 \leq n_2$, or $k_2 > n_2$. Suppose $k_2 \leq n_2$. Since $\{2^{k_3-n_3},3^{n_3}\}$ is a shift-winning coalition, then we have
\begin{equation}
\label{n3winning}   
(k_3-n_3)w_2 + n_3w_3 \geq 1.
\end{equation}  
(We note that $k_3-n_3 < k_2\le n_2$ in this case.)
 And if $k_2 > n_2$, then since $k_3-n_3 < k_2 < k_1 + n_2$, and $k_1 \leq n_1$, the coalition $\{1^{k_3-n_2-n_3},2^{n_2},3^{n_3}\}$ is a shift-minimal winning coalition, and we have
\begin{equation}
\label{big_n3winning}   
(k_3-n_2-n_3)w_1 + n_2w_2 + n_3w_3 \geq 1.
\end{equation}  
 \end{itemize}

\subsection{Two-level games}

\begin{theorem}
\label{twolevel}
Let $H=H_\exists({\bf n},{\bf k})$ be a two-level disjunctive hierarchical game  with no passers and no dummies. Then $H$ is roughly weighted but not weighted iff one of the following conditions is satisfied: 
\begin{enumerate}
\item[(i)] ${\bf k}=(2,4) $ with $n_1\ge 2$ and $n_2\ge 4$;
\item[(ii)]  ${\bf k}=(k,k+2)$, where $k>2$,  with $n_1\ge k$ and $n_2=4$.
\end{enumerate}
If $[1;w_1,w_2]$ is a rough voting representation for $H$, then $w_2=w_1/2$. Moreover, in case (i) we have $(w_1,w_2)=(\frac{1}{2},\frac{1}{4})$.
\end{theorem}

\begin{proof}
Let $[1;w_1,w_2]$ be a rough voting representation for $H$ and $M$ be its unique shift-maximal losing coalition. As we do not have passers we have $k_1>1$. We need to consider two cases: (i) $k_2\le n_2$ and (ii) $k_2>n_2$. In the first case, due to (\ref{n1winning}), we have $k_1w_1\ge 1$, $k_2w_2\ge 1$ and by Corollary~\ref{w(M)=q}  $w(M)=(k_1- 1)w_1+(k_2-k_1)w_2= 1$. If only we had $k_1w_1> 1$ or $k_2w_2> 1$ we could decrease $w_1$ or $w_2$ and make $w(M)<1$ in which case the game would be weighted. Hence $k_1w_1=k_2w_2=(k_1- 1)w_1+(k_2-k_1)w_2=1$. This implies $1/k_1+k_1/k_2= 1$.
Let $k_2-k_1=d$. Then $\frac{1}{k_1}+\frac{k_1}{k_1+d}= 1$, which is equivalent to $k_1+d= k_1d$ or $d= \frac{k_1}{k_1-1}$. It implies $1<d\le 2$ whence $d=2$ and $k_1= 2$. Thus we have only one solution: $k_1=2$ and $k_2=4$. This implies ${\bf w}=(w_1,w_2)=(\frac{1}{2},\frac{1}{4})$.

Let us consider the second case. Due to (\ref{shift_maximal_weight}), (\ref{n1winning}) and (\ref{n2winning})  ${\bf w}$ satisfy the inequalities ${k_1w_1\ge 1}$, ${(k_2-n_2)w_1+n_2w_2\ge 1}$ and the equality  $(k_1- 1)w_1+(k_2-k_1)w_2= 1$. The latter line must be a supporting line of the polyhedron area given by 
\[
k_1w_1\ge 1,\quad  (k_2-n_2)w_1+n_2w_2\ge 1,\quad  w_2\ge 0.
\]
 Indeed, if it cuts across this area, then  we will be able to find a point $(w_1,w_2)$ in this area with $(k_1- 1)w_1+(k_2-k_1)w_2< 1$. The game then will be weighted relative to $[1;w_1,w_2]$. This area has only  two extreme points and the line must pass through at least one of them. This is either when $w_2=0$ or when $k_1w_1=1$ and $(k_2-n_2)w_1+n_2w_2=1$.
Firstly, let us consider the case when $w_2=0$. In such a case $(k_2-n_2)w_1\ge 1$ and $(k_1-1)w_1= 1$. This can only happen when $k_2-n_2 \geq k_1-1$ or $n_2 \leq k_2-k_1+1$, but by Theorem~\ref{Thm_1} we cannot have $n_2 < k_2-k_1+1$, so it must be that $n_2 = k_2-k_1+1$. But in this case $H$ is weighted by Theorem~\ref{BTW}. 

Suppose now $k_1w_1=(k_2-n_2)w_1+n_2w_2=1$ and  $(k_1- 1)w_1+(k_2-k_1)w_2= 1$. Expressing $w_1$ and $w_2$ from the first two equations and substituting into the third we obtain $(k_2-k_1)(n_2-(k_2-k_1))= n_2$. Denoting $d=k_2-k_1$ we can rewrite this as $n_2= d+1+\frac{1}{d-1}$. As $n_2$ must be an integer we get $d=2$ and  $n_2= 4$. Now  
$
w_2=\frac{k_1-k_2+n_2}{n_2k_1}=\frac{1}{2k_1}=\frac{w_1}{2}.
$
It is easy to check that these weights indeed make $H$ roughly weighted and ${w(M)=(k-1)w_1+2\cdot\frac{w_1}{2}=kw_1=1}$.
\end{proof}

\subsection{Three-level games}

We now focus our investigation on the three-level games. The first observation that we are going to make is that if a three-level game disjunctive hierarchical does not have passers and dummies, then it is not weighted. This follows from Theorem~\ref{BTW}. We will use this often. 

\begin{proposition}
\label{w_3>0}
Let $H=H_\exists({\bf n},{\bf k})$ be a $3$-level  disjunctive hierarchical game without passers and dummies. If $H$ is roughly weighted with a rough voting representation $[1;w_1,w_2,w_3]$, then the subgame $H'=H_\exists({\bf n}',{\bf k}')$, where ${\bf n}'=(n_1,n_2)$ and ${\bf k}'=(k_1,k_2)$, is weighted.  
\end{proposition}

\begin{proof}
Suppose $H'$ is not weighted. Then by Lemma~\ref{lastweightzero} $w_3=0$ and $w_2>0$ by Proposition~\ref{allnonzero}. By Theorem~\ref{twolevel} we must consider the following two cases. 

(i) ${\bf k}=(2,4,a)$, where $a\ge 5$. If $a\le n_3$, we would have a winning coalition $\{3^{a}\}$ of zero weight and any player of the first two levels will be a passer. Thus $a>n_3$. As $H'$ falls under case (i) of Theorem~\ref{twolevel} we have $w_1=\frac{1}{2}$ and $w_2=\frac{1}{4}$ and $n_2\ge 4$. As no dummies present, by Theorem~\ref{Thm_1}, we have $n_3\ge k_3-k_2+1= a-4+1=a-3$. As $a-n_3\le 3<n_2$,  the coalition $\{2^{a-n_3}, 3^{n_3}\}$ is a winning coalition whose weight is at most $\frac{3}{4}$, a contradiction.

(ii) ${\bf k}=(k,k+2,a)$, where $k\ge 3$ and $a\ge k+3$. As $H'$ now is under case (ii) of Theorem~\ref{twolevel}
we know that $n_2= 4$ and $w_1=\frac{1}{k}$ and $w_2=\frac{1}{2k}$ with $k\ge 3$. As no dummies present, by Theorem~\ref{Thm_1}, we have $n_3\ge k_3-k_2+1= a-(k+2)+1=a-k-1$. As $a-n_3 \le k+1 \le n_1+1$, then $a-n_3 \le n_1+1$. So either the coalition $\{2^{a-n_3}, 3^{n_3}\}$ is a legitimate winning coalition (if $a-n_3 \le n_2=4$) or alternatively $\{1^{a-n_3-4}, 2^4, 3^{n_3}\}$ is a legitimate winning coalition. In the first case the weight of such winning coalition will be $\frac{a-n_3}{2k}\le \frac{2}{3}$. In the second, the weight of the winning coalition would be 
\[
\frac{a-n_3-4}{k}+\frac{4}{2k}=\frac{a-n_3-2}{k}\le \frac{k-1}{k}<1.
\]
In both cases the weight of such minimal winning coalition is less than 1, a contradiction.
\end{proof}
%
We now know that when $H=H_\exists({\bf n},{\bf k})$ is a $3$-level  disjunctive hierarchical game without passers and dummies, then the two-level game $H'=H_\exists({\bf n}',{\bf k}')$, where ${\bf n}'=(n_1,n_2)$ and ${\bf k}'=(k_1,k_2)$, is weighted. This restricts possible values of $k_1$ and $k_2$. Essentially, according to Theorem~\ref{BTW}, we have two cases. In one case $k_2=k_1+1$. Let us explore it.

\begin{proposition}
\label{all_are_smaller}
Let $H=H_\exists({\bf n},{\bf k})$ with no passers and no dummies, where ${\bf n}=(n_1,n_2,n_3)$ and ${\bf k}=(k,k+1,k+a)$, where $a\ge 2$ is a positive integer, with $k_i\le n_i\,$ for $i=1,2,3$. Then $H$ is not roughly weighted. 
\end{proposition}

\begin{proof}
The shift-minimal winning coalitions are $\{1^{k}\}$, $\{2^{k+1}\}$ and $\{3^{k+a}\}$ and the inequalities in this case will be $kw_1\ge 1$, $(k+1)w_2\ge 1$ and $(k+a)w_3\ge 1$, respectively. The shift-maximal equation in this case will be $(k-1)w_1+w_2+(a-1)w_3=1$. As in the proof of Theorem~\ref{twolevel}, we may assume that all three aforementioned inequalities are in fact equalities, that  is  $w_1=\frac{1}{k}$, $w_2=\frac{1}{k+1}$, and that $w_3=\frac{1}{k+a}$. Substituting these weights in the shift-maximal equation we get a contradiction as  $(k-1)\frac{1}{k}+\frac{1}{k+1}+\frac{a-1}{k+a}=1$ is equivalent to $\frac{1}{k+1}+\frac{a-1}{k+a}=\frac{1}{k}$ which never happens for $k\ge 2$ and $a\ge 2$ as this is equivalent to $(a-1)k^2+(a-2)k-a=0$ from which $a=\frac{k^2+2k}{k^2+k-1}<2$, a contradiction.
\end{proof}

The most basic type of disjunctive hierarchical games that we will be referring to constantly is the one with ${\bf k}=(2,3,4)$, We will chracterise these in Lemmata~\ref{234} and~\ref{23k}.

\begin{lemma}
\label{234}
Let ${\bf k}=(2,3,4)$. The 3-level game $G=H_\exists({\bf n},{\bf k})$ with no dummies is roughly weighted if and only if $n_1 \geq 2$ and one of the following is true:  
\begin{itemize}
\item [(i)] ${\bf n}=(n_1,2,n_3)$, where $n_3 \geq 3$ and ${\bf w}=(\frac{1}{2},\frac{1}{4},\frac{1}{4})$;  
\item [(ii)] ${\bf n}=(n_1,n_2,2)$, where $n_2 \geq 3$ and ${\bf w}=(\frac{1}{2},\frac{1}{2}- \alpha,\alpha)$ with $\alpha \in [0,\frac{1}{6}]$;
\item [(iii)] ${\bf n}=(n_1,2,2)$ and ${\bf w}=(\frac{1}{2},\frac{1}{2}- \alpha,\alpha)$ with $\alpha \in [0,\frac{1}{4}]$. 
\end{itemize}
\end{lemma}

\begin{proof}
Firstly, we note that by Theorem~\ref{Thm_1}  we have $2=k_1\le n_1$. We also note that the shift-maximal equation in this case by~(\ref{shift_maximal_weight}) is in this case
\begin{equation}
\label{shift_maximal_i}
w_1+w_2+w_3 = 1.
\end{equation}

\noindent
\textbf{Case (i)}. $k_i\le n_i$ for $i=1,2,3$ is considered in Proposition~\ref{all_are_smaller}. There are no solutions in this case.\par\smallskip

\noindent\textbf{Case (ii)}. Suppose $k_2 > n_2, k_3 \leq n_3$, then $n_3 \geq 4$. As $k_2=3$, by Corollary~\ref{n_2>1} it follows that $n_2$ must be $2$. The shift-minimal winning coalitions are $\{1^2\}$, $\{1,2^2\}$, $\{3^{4}\}$. So the corresponding inequalities  are $w_1 \geq \frac{1}{2}$, $w_1 + 2w_2 \geq 1$, and  $w_3 \geq \frac{1}{4}$. As in the proof of Proposition~\ref{all_are_smaller}, we may assume $w_3 = \frac{1}{4}$. From~(\ref{shift_maximal_i}) we get $w_1+w_2 = \frac{3}{4}$. It follows that $w_1 + 2(\frac{3}{4}-w_1) \geq 1$, whence $w_1 \leq \frac{1}{2}$, forcing $w_1 = \frac{1}{2}, w_2 = \frac{1}{4}, w_3 = \frac{1}{4}$. So it is roughly weighted only when ${\bf n}=(n_1,2,n_3)$, and ${\bf w}=(\frac{1}{2},\frac{1}{4},\frac{1}{4})$, as required.  
\par\smallskip

\noindent\textbf{Case (iii)}. Suppose $k_2 \leq n_2$, $k_3>n_3$. Then $n_3 \leq 3$ and $n_2 \geq 3$. Then the shift-minimal winning coalitions are $\{1^2\}$, $\{2^{3}\}$, $\{2^{4-n_3},3^{n_3}\}$. To justify this we have to note that by Theorem~\ref{Thm_1} $k_3-n_3< k_2\le n_2$ whence $4-n_3\le n_2$ and the last coalition is legitimate. The inequalities then will be $w_1 \geq \frac{1}{2}, w_2 \geq \frac{1}{3}$ and $(4-n_3)w_2 + n_3w_3 \geq 1$.
As above we may assume $w_1 = \frac{1}{2}$. Substituting this value of $w_1$ into~(\ref{shift_maximal_i}) in this case we get $w_2+w_3 = \frac{1}{2}$. So  $(4-n_3)w_2 + n_3(\frac{1}{2}-w_2) \geq 1$. Now by Corollary~\ref{n_2>1}  $n_3$ is either 2 or 3. If it is 3, then we get $w_2 + 3(\frac{1}{2}-w_2) \geq 1$ giving $w_2 \leq \frac{1}{4}$, but we know that $w_2 \geq \frac{1}{3}$, contradiction. If it is 2, then the system has solutions for any $w_2 \geq \frac{1}{3}$ and the game is roughly weighted with ${\bf w}=(\frac{1}{2},\frac{1}{2}-\alpha,\alpha)$, where $\alpha \in [0,\frac{1}{6}]$. In this case, ${\bf n}=(n_1,n_2,2)$. 
\par\smallskip
\noindent\textbf{Case (iv)}. Suppose $k_2 > n_2$ and $k_3>n_3$. Then $n_3 \leq 3$ and $n_2=2$. Since by Theorem~\ref{Thm_1} and Corollary~\ref{n_2>1} we have $4 = k_3 < k_2 + n_3$, then $4 - n_3 \leq 2 = n_2$ and the shift-minimal winning coalitions are $\{1^2\}$, $\{1,2^{2}\}$, $\{2^{4-n_3},3^{n_3}\}$. Then the inequalities will be: $w_1\ge \frac{1}{2}$, $w_1+2w_2\ge 1$, $(4-n_3)w_2 + n_3w_3 \geq 1$. If $n_3=2$, then the latter becomes $2w_2+2w_3 \geq 1$ or $w_2+w_3 \geq \frac{1}{2}$ which, in particular, imply $w_1=1-w_2-w_3\le \frac{1}{2}$, whence $w_1 = \frac{1}{2}$ and $w_2+w_3=\frac{1}{2}$.  Now from $w_1+2w_2 \geq 1$ we get $w_2 \geq \frac{1}{4}$.  
So this gives the solution $n=(n_1,2,2)$ with $w=(\frac{1}{2},\frac{1}{2}- \alpha, \alpha)$, where $\alpha \in [0,\frac{1}{4}]$, as required. Now if $n_3=3$, then the inequalities will be $w_1\ge \frac{1}{2}$, $w_1+2w_2 \geq 1$, $w_2+3w_3 \geq 1$. Again substituting $w_3=1-w_1-w_2$ into the latter inequality gives $3w_1 + 2w_2 \leq 2$. As $2w_2\ge 1-w_1$ we get $1+2w_1\le 2$ and  $w_1 \leq \frac{1}{2}$. Hence $w_1 = \frac{1}{2}$ and $w_2+w_3=\frac{1}{2}$. This together with $w_2+3w_3 \geq 1$ gives $w_3 \geq \frac{1}{4}$. But since $w_1=\frac{1}{2}$, then $w_1+2w_2 \geq 1$ gives $w_2 \geq \frac{1}{4}$, and so $w_3 \leq \frac{1}{4}$. Thus we have the weights $w=(\frac{1}{2},\frac{1}{4},\frac{1}{4})$, with $n=(n_1,2,3)$. This works. 
\end{proof}

\begin{lemma}
\label{23k}
Let $H=H_\exists({\bf n},{\bf k})$, where ${\bf k}=(2,3,k)$ and $k\ge 5$, be a 3-level disjunctive hierarchical game with no dummies.  Then $n_1\ge 2$ and it is roughly weighted if and only if 
$
n_3=k-2,\ \text{and ${\bf w}=(\frac{1}{2},\frac{1}{2},0)$}.  
$
\end{lemma}

\begin{proof}
We note that the absence of dummies means that $k_2+n_3>k_3$ or $n_3>k-3$.
The shift-maximal equation is now
\begin{equation}
\label{shift_maximal_ii}
w_1+w_2+(k-3)w_3 = 1.
\end{equation}
\noindent \textbf{Case (i)}. The case when $k_i \leq n_i$ for all $i$ is treated in Proposition~\ref{all_are_smaller}. There are no solutions in this case.\par\smallskip

\noindent\textbf{Case (ii)}. As $k_1\le n_1$, suppose $k_2 > n_2$ and $ k_3 \leq n_3$.  It follows that $n_2$ must be $2$. The shift-minimal winning coalitions then are $\{1^2\}$, $\{1,2^2\}$, $\{3^{k}\}$. So the corresponding inequalities are $w_1 \geq \frac{1}{2}$, $w_1 + 2w_2 \geq 1$ and $w_3 \geq \frac{1}{k}$. By the usual trick we may assume that $w_3 = \frac{1}{k}$. Then from the shift-maximal equation~(\ref{shift_maximal_ii}) we get $w_1+w_2 = \frac{3}{k}$. It follows that $w_1 + 2(\frac{3}{k}-w_1) \geq 1$, so $w_1 \leq \frac{6-k}{k}\le \frac{1}{k}$, but $w_1 \geq \frac{1}{2}$, contradiction. 
\par\smallskip
\noindent\textbf{Case (iii)}. Suppose $k_2 \leq n_2$, and $n_3 < k_3=k$. Then $k_3-k_2+1=k-2\le n_3\le k-1$ and, in particular, by Corollary~\ref{n_2>1} $k-n_3\le 2\le n_2$.  Then the shift-minimal winning coalitions are $\{1^2\}$, $\{2^3\}$, $\{2^{k-n_3},3^{n_3}\}$, giving the inequalities $w_1 \geq \frac{1}{2}$, $w_2 \geq \frac{1}{3}$ and $(k-n_3)w_2 + n_3w_3 \geq 1$. We may set $w_1 = \frac{1}{2}$ which implies $w_2+(k-3)w_3 = \frac{1}{2}$. Let us consider two cases: (a) $n_3=k-1$ and (b) $n_3=k-2$.

(a) In this case the two inequalities become $w_2+(k-3)w_3 = \frac{1}{2}$ and $w_2+(k-1)w_3 \ge 1$. These imply $w_3\ge \frac{1}{4}$. But then $w_2+(k-3)w_3 \ge \frac{1}{3}+\frac{k-3}{4} >\frac{1}{2}$, contradiction.

(b) In this case the two inequalities become $w_2+(k-3)w_3 = \frac{1}{2}$ and $2w_2+(k-2)w_3 \ge 1$. This implies that either $w_3=0$ or $2(k-3)\le k-2$. The latter implies $k\le 4$, hence the only solution in this case is ${\bf w}=(\frac{1}{2},\frac{1}{2},0)$.\par\smallskip
 
\noindent\textbf{Case (iv)}. Suppose $k_2 > n_2$, $n_3 <k_3=k$, so $n_2=2$ and, as in case (iii),  $k-2\le n_3\le k-1$. Then the shift-minimal winning coalitions are $\{1^2\}$, $\{1,2^2\}$, $\{2^{k-n_3},3^{n_3}\}$, giving the inequalities $w_1\ge \frac{1}{2}$, $w_1+2w_2\ge 1$ and $(k-n_3)w_2 + n_3w_3 \geq 1$.  We have either (a) $n_3=k-2$ or $n_3=k-1$. 

(a) Suppose $n_3=k-2$. Then the last inequality becomes $2w_2+(k-2)w_3 \geq 1$. From the shift-maximal equation \eqref{shift_maximal_ii}  we get $w_2+(k-3)w_3\le \frac{1}{2}$, which together with the previous inequality implies either $w_3=0$ or $2(k-3)\le k-2$. As the latter implies $k\le 4$, we again have the solution ${\bf w}=(\frac{1}{2},\frac{1}{2},0)$.\par\smallskip

(b) Suppose $n_3=k-1$. Then the last inequality becomes $w_2+(k-1)w_3 \geq 1$. From the shift-maximal equation \eqref{shift_maximal_ii} we get $w_2+(k-3)w_3\le \frac{1}{2}$ from which $w_3\ge \frac{1}{4}$. But this contradicts to \eqref{shift_maximal_ii} since $w_1+w_2+(k-3)w_3\ge \frac{1}{2}+w_2+\frac{k-3}{4}>1$ for any $k\ge 5$.  
\end{proof}

If the disjunctive hierarchical game $H=H_\exists({\bf n},{\bf k})$ is roughly weighted, then all its reduced games will be also roughly weighted by Lemma~\ref{subs_and_reducs} and Proposition~\ref{allnonzero}.  Let us now make an important observation about the weights in those reduced games. Suppose $H$ has a rough voting representation $[1;w_1, w_2, w_3]$ and let $A=\{1^{s_1}, 2^{s_2},3^{s_3}\}$ be a submultiset with the total weight $w(A)=s_1w_1+s_2w_2+s_3w_3$. Then by Lemma~\ref{subs_and_reducs}, the reduced game $H^A$ has rough voting representation $[1 - w(A); w_1, w_2, w_3]$ or after normalisation
\begin{equation}
\label{normalisation}
\left[1; \frac{w_1}{1 - w(A)}, \frac{w_2}{1 - w(A)}, \frac{w_3}{1 - w(A)}\right].
\end{equation} 

\begin{lemma}
\label{k_1=2}
Suppose that a $3$-level hierarchical game $H=H_\exists({\bf n},{\bf k})$, where ${\bf k}=(2,k_2,k_3)$, has no  dummies and is roughly weighted with rough voting representation $[1;w_1,w_2,w_3]$. Then either $w_3=0$ or ${\bf k}=(2,3,4)$.
\end{lemma}

\begin{proof}
Suppose $w_3>0$. If $k_2>3$, then $n_2\ge k_2-k_1+1=k_2-1$ so the second level contains at least $k_2-1$ elements and, in particular, $A=\{2^{k_2-3}\}$ is a submultiset of the multiset of players. Let us consider the reduced game $H^A$. Then by Proposition~\ref{remove_one}, $H^A=H({\bf n}',{\bf k}')$ with ${\bf n'}=(n_1,n_2-k_2+3, n_3)$ and ${\bf k}'=(2,3,k_3-k_2+3)$. Since $n_2-k_2+3\ge 2$ the reduced game still has three levels. By Proposition~\ref{subs_and_reducs} the game $H^A$ is also roughly weighted and, due to (\ref{normalisation}) the last weight of it will still be nonzero. Having $k_3 - k_2 + 3 > 4$ would imply by Lemma~\ref{23k} that the last weight is zero. Since that is not the case, then we have $k_3-k_2+3=4$ so ${\bf k}'=(2,3,4)$. 

Let $w=w(A)=(k_2-3)w_2$. Then by Lemma~\ref{234} and  (\ref{normalisation}) we have $\frac{w_1}{1-w}=\frac{1}{2}$. This means that $2w_1=1-w<1$ which contradicts the fact that $\{1^2\}$ is a wining coalition in $H$. Hence $k_2=3$ and ${\bf k}=(2,3,4)$.
\end{proof}

\begin{corollary}
\label{24x}
There does not exist a roughly weighted $3$-level disjunctive hierarchical game $H=H_\exists({\bf n},{\bf k})$ with ${\bf k}=(2,4,k_3)$ and  no  dummies.
\end{corollary}

\begin{proof}
Suppose on the contrary that $H$ is roughly weighted with rough voting representation $[1;w_1,w_2,w_3]$. Then by Lemma~\ref{k_1=2} we must have $w_3=0$. Consider $H'=H_\exists({\bf n}',{\bf k}')$, where ${\bf n}'=(n_1,n_2)$ and ${\bf k}'=(k_1,k_2)$. If it is weighted, then by Theorem~\ref{BTW} $n_2=k_2-k_1+1=3$. And if it is not, then $n_2 \geq 4$ by Theorem~\ref{twolevel}. In either case we have shift-minimal winning coalitions $\{1^2\}$ and $\{1,2^3\}$, hence  $w_1\ge \frac{1}{2}$ and $w_1+3w_2\ge 1$. By Theorem~\ref{Thm_1} we have $k_3-n_3\le 3$ so the third shift-minimal winning coalition is of the type $\{2^{k_3-n_3},3^{n_3}\}$. The weight of such coalition is not greater than $3w_2$. So we must have $w_2\ge \frac{1}{3}$. At the same time from the shift-maximal equation $w_1+2w_2=1$ and $w_1\ge \frac{1}{2}$ we have $w_2\le \frac{1}{4}$. This is a contradiction. 
\end{proof} 

\begin{lemma}
\label{k1k2}
Suppose that a $3$-level disjunctive hierarchical game $H=H_\exists({\bf n},{\bf k})$ with no passers and no dummies is roughly weighted. Then $k_2 - k_1 = 1$.
\end{lemma}

\begin{proof}
Let $[1;w_1,w_2,w_3]$ be  a rough voting representation of $G$. As there is no passers, $k_1\ge 2$.  Suppose $k_2 - k_1 \geq 2$. 
Observe that all $3$-level games with $k_2 - k_1 \geq 2$ can be reduced to a $3$-level game where $k_1=2$, $k_2=4$ as follows. First, take the reduced game $H_1=H^{A}$ with $A=\{1^{k_1-2}\}$, which will result in a game $H_1=H_\exists(\textbf{n}',\textbf{k}')$, with $\textbf{n}'=(n_1',n_2,n_3)$ and $\textbf{k}'=(2,k_2',k_3')$, where $k_2'=k_2-k_1+2$, $k_3'=k_3-k_1+2$ and $n_1'=n_1-k_1+2$. 
Since $H_1$ is roughly weighted, then by Corollary~\ref{24x} we have $k_2' \geq 5$. By Theorem~\ref{Thm_1} $n_2 \ge k_2-k_1= k_2'-2$ and $n_2-(k_2'-4)\ge 2$. This shows that $n_2$ has enough players for a further reduction to $H_1^{A'}$, where $A'=\{2^{k_2'-4}\}$, without collapsing the second level. The resulting game with $\textbf{k}''=(2,4,k_3'-(k_2'-4))$ is not roughly weighted by Corollary~\ref{24x} which proves the lemma. 
\end{proof}

By combining Lemmata~\ref{23k} and~\ref{k1k2} we get the following

\begin{corollary}
\label{which_games}
If a $3$-level disjunctive hierarchical game $H=H_\exists({\bf n},{\bf k})$ does not have passers and dummies and is roughly weighted, then it belongs to one of the following two categories: 
\begin{itemize}
\item [(i)] $\textbf{k}=(k,k+1,k+2)$;
\item [(ii)] $\textbf{k}=(k,k+1,k_3)$, such that  $n_3 = k_3-k\ge 3$.
\end{itemize}
\end{corollary}

\begin{proof}
By Lemma~\ref{k1k2} we have $k_2=k_1+1$. To prove the other claims we make a reduction of $H$ and consider $H'=G^A$ with $A=\{1^{k_1-2}\}$. Then $H'$ has parameters ${\bf n'}=(n_1-k_1+2,n_2,n_3)$ and ${\bf k'}=(2, k_2-k_1+2, k_3-k_1+2)=(2, 3, k_3-k_1+2)$. Now either $k_3'=4$, and in this case $k_3=k+2$, or by Lemma~\ref{23k} $n_3=k_3'-2=k_3-k_1$. Since in the latter case we have $k_3'\ge 5$, then we get $k_3-k\ge 3$. 
\end{proof}

So it is these two categories of games that we need to analyse. They will be analysed in the following two lemmas. We refer in their study to Lemmas~\ref{234} and~\ref{23k}.

\begin{lemma}
\label{general234}
A $3$-level game $H=H_\exists({\bf n},{\bf k})$ with ${\bf k}=(k,k+1,k+2)$ and $k \geq 3$ is roughly weighted  if and only if $n_1\ge k$ and one of the following conditions is satisfied:
\begin{enumerate}
\item[(a)] \text{${\bf n}=(n_1,2,2)$, ${\bf w}=\left(\frac{1}{k},\frac{1-2\alpha}{k},\frac{2\alpha}{k}\right)$, where $\alpha \in \left[0,\frac{1}{4}\right]$;}
\item[(b)] \text{${\bf n}=(n_1,n_2,2)$, $n_2\ge 3$, and ${\bf w}=\left(\frac{1}{k},\frac{1}{k},0\right)$.}
\end{enumerate}
\end{lemma}

\begin{proof}
Firstly, it is easy to check that the games in (a) and (b) are indeed, roughly weighted with the specified set of weights.

If $H$ is roughly weighted, then upon reducing it to $H^A=H_\exists({\bf n}',{\bf k}')$, where $A=\{1^{k-2}\}$, we get $\textbf{k}'=(2,3,4)$. Also, $\textbf{n}'$ must fall into one of the three cases given in Lemma~\ref{234}. Let us  analyze them one by one. 
\par
\textbf{Case (i)}. ${\bf n}'=(n_1',2,n_3)$, where $n_1'\ge 2$, $n_3 \geq 3$, and ${\bf w}=(\frac{1}{2},\frac{1}{4},\frac{1}{4})$. As mentioned earlier, the voting representation of $H^A$ will then be
\[
\left[1; \frac{w_1}{1 - w_1(k-2)}, \frac{w_2}{1 - w_1(k-2)}, \frac{w_3}{1 - w_1(k-2)}\right]
\]
and it has to match the voting representation of the reduced game at hand, namely $[1;\frac{1}{2},\frac{1}{4},\frac{1}{4}]$. It follows that $\frac{w_1}{1 - w_1(k-2)} = \frac{1}{2}$, so $w_1 = \frac{1}{k}$. Also, $\frac{w_2}{1 - w_1(k-2)} =\frac{w_3}{1 - w_1(k-2)}= \frac{1}{4}$, giving  $w_2 =w_3= \frac{1}{2k}$. Now we test to see if these weights indeed define the original hierarchical game $H$. Now since $(k+2)w_3 = \frac{k+2}{2k}$ is never greater than or equal to 1 for $k \geq 3$ coalition $\{3^{k_3}\}$ does not exist in $H$ (otherwise it would be winning), that is, $n_3\le k+1$.  Also $w_2+(k+1)w_3=2w_2+kw_3 =\frac{k+2}{2k}$,
and a shift-minimal winning coalition $\{2^{i},3^{k+2-i}\}$, for $i=1,2$,  also does not exist in $H$ for $k\ge 3$. 

So the coalition $\{1^{(k+2)-n_2-n_3},2^{n_2},3^{n_3}\}$ must be a shift-minimal winning coalition in $H$. So its weight should be at least the threshold, which is 1, i.e., $\frac{k+2-2-n_3}{k} + \frac{2}{2k} + \frac{n_3}{2k} = \frac{2k-n_3+2}{2k} \geq 1$. But this is never true for $n_3 \geq 3$. Therefore $H$ is not roughly weighted in this case.
\par\smallskip

\textbf{Case (ii)}. ${\bf n}'=(n_1',n_2,2)$, where $n_1'\ge 2$, $n_2 \geq 3$, and ${\bf w}=(\frac{1}{2},\frac{1}{2}-\alpha,\alpha)$ for some $\alpha \in [0,\frac{1}{6}]$. Here $w_1$ is still $\frac{1}{k}$. But $\frac{w_2}{1 - w_1(k-2)} = \frac{1}{2}-\alpha$. It follows that $\frac{kw_2}{2} = \frac{1}{2}-\alpha$, so that $w_2 = \frac{1-2\alpha}{k}$. Also, $\frac{w_3}{1 - w_1(k-2)} = \alpha$, implying  $w_3 = \frac{2\alpha}{k}$. 
As we do not have dummies there must be a winning coalition consisting of $k+2$ players. This would be either  $\{2^k,3^2\}$ or $\{1^{k-n_2},2^{n_2},3^2\}$ depending on what is greater $n_2$ or $k$. But $w(\{2^k,3^2\}) = k\frac{1-2\alpha}{k} + 2\frac{2\alpha}{k} = 1-2\alpha+\frac{4\alpha}{k}$. For this to be winning we need $1-2\alpha+\frac{4\alpha}{k} \geq 1$, giving $\frac{4\alpha}{k} \geq 2\alpha$. So either we have $\frac{2}{k} \geq 1$ or $\alpha=0$. As $k>2$ we have the latter. Thus $\{2^k,3^2\}$ can be winning only for $\alpha =0$ in which case ${\bf w}=(\frac{1}{k},\frac{1}{k},0)$. \par
Suppose now that the winning coalition consisting of $k+2$ players is $\{1^{k-n_2},2^{n_2},3^2\}$. Its weight then is $\frac{k-n_2}{k} + \frac{n_2(1-2\alpha)}{k} + \frac{4\alpha}{k} \geq 1$. It follows that  $2\alpha  \geq \alpha n_2$, whence $\alpha=0$. In both cases we have (b). \par\smallskip

\textbf{Case (iii)}. ${\bf n}'=(n_1',2,2)$, where ${\bf w}=(\frac{1}{2},\frac{1}{2}-\alpha,\alpha)$, and $\alpha \in [0,\frac{1}{4}]$. This gives us case~(a).
\end{proof} 

Now to the remaining case.

\begin{lemma}
\label{general_k2k3}
Any 3-level game $G=H({\bf n},{\bf k})$, where ${\bf k}=(k,k+1,k_3)$ such that $k_3 - (k+1) \geq 2$ and $G$ has no passers and no dummies, is roughly weighted if and only if the following is true.
\begin{itemize}
\item [(i)] ${\bf n}=(n_1,n_2,n_3)$, where $n_3=k_3 - k \geq3$, and ${\bf w}=(\frac{1}{k},\frac{1}{k},0)$.
\end{itemize}
\end{lemma}

\begin{proof} 
Upon reducing $G$ to $G^A=H'({\bf n}',{\bf k}')$, where $A=\{1^{k-2}\}$, we get $\textbf{k}'=(2,3,k_3')$, where $k_3' \geq 5$. If the game $G^A$ is roughly weighted, then by Lemma~\ref{23k} it has to have $\textbf{n}=(n_1,n_2,k_3'-2)$, where $n_2 \geq 2$, and $n_1 \geq 2$, and the weights consistent with ${\bf w}=(\frac{1}{2},\frac{1}{2},0)$. So we get $\frac{w_1}{1 - w_1(k-2)} = \frac{1}{2}$, so $w_1 = \frac{1}{k}$. Also, $\frac{w_2}{1 - w_1(k-2)} = \frac{1}{2}$, meaning $\frac{w_2}{1 - w_1k + 2w_1} = \frac{1}{2}$, so $\frac{w_2}{1 - 1 + 2\frac{1}{k}} = \frac{1}{2}$. Therefore $w_2 = \frac{1}{k}$, and $w_3 = 0$. It can be easily checked that these weights give a valid hierarchical game where $\textbf{n}=(n_1,n_2,k_3-k), w=(\frac{1}{k},\frac{1}{k},0)$. 
\end{proof}

\subsection{Main Results}

Finally we are ready to collect all facts together and give a full characterisation of roughly weighted disjunctive hierarchical games. As the classification of weighted hierarchical games is already given in Theorem~\ref{BTW} we will characterize only nonweighted ones. Then by duality we will derive similar results for conjunctive hierarchical games.
 
\begin{theorem}
\label{RW_HTAS}
If $H=H_\exists(\textbf{n},\textbf{k})$ is an $m$-level nonweighted hierarchical game, then it is roughly weighted if and only if one of the following is true:
\begin{itemize}
\item[(i)] $k_1=1$ and $m$ is arbitrary;
\item[(ii)] ${\bf k}=(2,4) $ with $n_1\ge 2$ and  $n_2\ge 4$;
\item[(iii)]  ${\bf k}=(k,k+2)$, with $n_1 \ge k > 2$ and $n_2=4$;
\item[(iv)] ${\bf k}=(2,3,4)$ and ${\bf n}=(n_1,2,n_3)$ with $n_1\ge 2, n_3\ge 3$;
\item[(v)] ${\bf k}=(k,k+1,k+2)$, and  ${\bf n}=(n_1,2,2)$, where $2 < k \le n_1$ or ${\bf n}=(n_1,n_2,2)$ with $2 < k \le n_1$ and $n_2\ge 3$;
\item[(vi)] ${\bf k}=(k,k+1,k_3)$, ${\bf n}=(n_1,n_2,n_3)$ such that $2\le k\le n_1$, and $n_3 = k_3 - k\ge 3$.
\item[(vii)] $k_m= k_{m-1}+n_m$ and the subgame $H_\exists({\bf n}',{\bf k}')$, where ${\bf n}'=(\row n{m-1})$ and ${\bf k}'=(\row k{m-1})$, falls under one of the types (i)--(vi).
\end{itemize}
\end{theorem}

\begin{proof} 
Firstly, we note that if $k_m\ge k_{m-1}+n_m$, then the $m$th level of players consists of dummies. This game will be roughly weighted if and only if the game $H_\exists({\bf n}',{\bf k}')$ is. This situation is described in (vii). So we consider that $k_m< k_{m-1}+n_m$.  If $k_1=1$, then we can set $w_1=1$ and $w_i=0$ for $i>1$. Suppose $k_1>1$. Then $H$ does not have passers and dummies so by Lemma~\ref{fourlevels} $m\le 4$.

If $m=2$  the result follows from Theorem~\ref{twolevel}. If  $m=3$ the result follows from Lemmata of the previous section. We will now show that the fourth level cannot be added without introducing dummies. Suppose that $H$ has the fourth level whose players are not dummies. We know from Corollary~\ref{L4} and Proposition~\ref{allnonzero} that $w_3\ne 0$ and $w_4=0$. 

By letting $A=\{1^{k_1-2}\}$ and considering the reduced game $H'=H^A$ we obtain a 4-level game $H'=H_\exists({\bf n}',{\bf k}')$ which is roughly weighted by Lemma~\ref{subs_and_reducs} and Proposition~\ref{allnonzero}. We can now consider the subgame $H''=H_\exists({\bf n}'',{\bf k}'')$ of $H'$, where ${\bf n}''=(n_1',n_2',n_3')$ and ${\bf k}''=(k_1',k_2',k_3')$. It will be again roughly weighted by Lemma~\ref{subs_and_reducs}  and we may apply Lemma~\ref{k_1=2} to this 3-level game. By Lemma~\ref{subs_and_reducs} all weights of this game will be nonzero. By Lemma~\ref{k_1=2} we will then have $\textbf{k}''=(2,3,4)$ and thus
 $\textbf{k}'=(2,3,4,k_4)$, $\textbf{n}'=(n_1',n_2,n_3,n_4)$, where $n_1'=n_1-k_1+2$. 
As there are no dummies in $H$, by Theorem~\ref{Thm_1} we have $n_4 \geq k_4-4+1=k_4-3$. Thus either $\{3^3,4^{k_4-3}\}$ or $\{2,3^2,4^{k_4-3}\}$ is a winning coalition as total number of players in each is $k_4$. By Lemma~\ref{234}, in all cases when $H'$ is roughly weighted we have $w_2+w_3=\frac{1}{2}$ and $w_3\le \frac{1}{4}$. However, the weight of the coalition $\{2,3^2,4^{k_4-3}\}$ is $w_2+2w_3=\frac{1}{2}+w_3\le \frac{3}{4}$. The same is true for the coalition $\{3^3,4^{k_4-3}\}$, giving a contradiction.
\end{proof}

And the following gives the full characterisation of roughly weighted conjunctive simple games.

\begin{theorem}
\label{RW_HTAS_con}
If $H=H_\forall(\textbf{n},\textbf{k})$ is an $m$-level nonweighted conjunctive hierarchical game. Then it is roughly weighted if and only if one of the following is true:
\begin{itemize}
\item[(i)] $k_1=n_1$ and $m$ is arbitrary;
\item[(ii)] ${\bf k}=(n_1-1,n_1+n_2-3)$, where $n_1\ge 2$, $n_2 \geq 4$;
\item[(iii)]  ${\bf k}=(k,k+2)$, with $1 \leq k < n_1 - 1$, and $n_2=4$;
\item[(iv)] ${\bf k}=(n_1-1,n_1,n_1+n_3-1)$ and ${\bf n}=(n_1,2,n_3)$ such that $n_1 \geq 2, n_3 \geq 3$;
\item[(va)] ${\bf k}=(k,k+1,k+2)$, where $1 \leq k < n_1 - 1$ and ${\bf n}=(n_1,2,2)$; 
\item[(vb)] ${\bf k}=(k_1,k_2,k_2+1)$ with ${\bf n}=(n_1,n_2,2)$, where $k_2-k_1=n_2-1$, $1 \leq k_1 < n_1 - 1$ and $n_3 \geq 3$;
\item[(vi)] ${\bf k}=(k_1,k_2,k_2+1)$, with  $1 \leq k \le n_1 - 1$ and $n_2 \geq 3$;
\item[(vii)] $k_{m-1}= k_m$, and the subgame $H_\forall({\bf n}',{\bf k}')$, where ${\bf n}'=(\row n{m-1})$ and ${\bf k}'=(\row k{m-1})$, falls under one of the types (i)--(vi).
\end{itemize}
\end{theorem}

\begin{proof}
The proof will consists of calculating the duals for the games listed in Theorem~\ref{RW_HTAS}. \par\smallskip

(i) In the dual game by Theorem~\ref{RW_HTAS} we have $k_1=1$, so in this game we have $n_1-k_1+1=k^*_1$, meaning $k_1^*=n_1$.\par\smallskip 

(ii) Here we have $k^*_1=n_1-k_1+1$ where $k_1=2$, so $k^*_1=n_1-1$. Also $k^*_2=n_1+n_2-k_2+1=n_1+n_2-3$. \par\smallskip

(iii) In the dual game $k^*_1=n_1-k+1$ and $k^*_2=n_1+n_2-(k+2)+1=n_1-k+3=k^*_1+2$. As $n_1\ge k>2$, we have $n_1-1 > k^*_1 \geq 1$.\par\smallskip

(iv) Since ${\bf k}=(2,3,4)$ in the dual game, then this gives $n_1-k_1+1=k^*_1$, so that $k^*_1=n_1-1$. Also, $n_1+n_2-k_2+1=k^*_2$, giving $k^*_2=n_1+n_2-2=n_1$ since $n_2=2$. We also get $k^*_3=n_1+n_2+n_3-k_3+1=n_1+n_3-1$. \par\smallskip

(v) We have two cases here so we treat them separately.
\begin{enumerate}
\item[(a)] $k^*_1=n_1-k+1$, $k^*_2=n_1+n_2-(k+1)+1=n_1-k+2=k^*_1+1$,  and  $k^*_3=n_1+n_2+n_3-(k+2)+1=n_1-k+3=k^*_1+2$.
\item[(b)] $k^*_1=n_1-k+1$, $k^*_2=n_1+n_2-(k+1)+1=n_1+n_2-k$, and $k^*_3=n_1+n_2+n_3-(k+2)+1=n_1+n_2-k+1=k^*_2+1$.
\end{enumerate}

(vi) $k^*_1=n_1-k_1+1$, $k^*_2=n_1+n_2-(k_1+1)+1=n_1+n_2-k_1=k^*_1+n_2-1$,  $k^*_3=n_1+n_2+n_3-k_3+1=n_1+n_2-k_1+1=k^*_2+1$.\par\smallskip

(vii) We have $n_m-k_m=k_{m-1}$ and this implies $k^*_m=k^*_{m-1}$.
\end{proof}

\section{Conclusion and Further Research}

This paper provides a complete characterization of roughly weighted hierarchical gams both disjunctive and conjunctive. As we have seen weighted hierarchical games can have only up to two nontrivial levels and roughly weighted only up to three levels. So in general, hierarchical games are rather far from weighted ones.  \citeA{GHS} introduced three hierarchies of simple games, each depend on a single parameter and for each hierarchy the union of all classes is the whole class of simple games.  One idea that was suggested is to generalize roughly weighted games as follows. Rough weightedness allows just one value of the threshold $q=1$ (after normalization), where coalitions of weight 1 can be both losing and winning. Instead of just single threshold value we may allow values of thresholds from a certain interval $[1,a]$ to possess this property, that is, coalitions whose weight is between $1$ and $a$ can be both winning or losing. They denote this class of games ${\cal C}_a$. The question which deserves further study is how big should be $a$ so that  all hierarchical $n$-level games are in ${\cal C}_a$.

\bibliographystyle{spmpsci}
\bibliography{short}


\end{document}